\documentclass[11pt]{amsart}
\usepackage{amscd,amsfonts,amsmath,amssymb,amstext,amsthm, mathrsfs}
\usepackage{setspace}
\date{}
\def\dim{\operatorname{dim}}

\def\Aut{\operatorname{Aut}}

\input xy
\xyoption{all}
\theoremstyle{plain}
\newtheorem{theorem} {Theorem} [section]

\newtheorem{proposition}[theorem]{Proposition}

\newtheorem{conjecture}{Conjecture}
\theoremstyle{definition}

\newtheorem{remark}[theorem] {Remark}

\begin{document}

 \title {Simplicity of tangent bundles of smooth horospherical varieties of Picard number one}
 \author [Jaehyun Hong]{Jaehyun Hong}
 \address{Center for Complex Geometry,  Institute for Basic Science (IBS), Daejeon 34126, Republic of Korea}
   \email{jhhong00@ibs.re.kr}

 \maketitle

\noindent{\small {\bf Abstract.}
Recently, Kanemitsu has discovered a counterexample to the long-standing conjecture that the tangent bundle of a Fano manifold of Picard number one is (semi)stable. His counterexample is a smooth horospherical variety. There is a weaker conjecture that the tangent bundle of a Fano manifold of Picard number one is simple.

We prove that  this weaker conjecture is valid for  smooth horospherical varieties of Picard number one.  Our proof follows from the existence of an irreducible family of unbendable rational curves  whose  tangent vectors span the tangent spaces of the horospherical variety at general points.} \\

  \noindent {\small {\bf R\'esum\'e.}
R\'ecemment, Kanemitsu a d\'ecouvert un contre-exemple \`a la conjecture de longue date selon laquelle le faisceau tangent d'une variété de Fano de nombre de Picard   un est (semi) stable. Son contre-exemple est une vari\'et\'e horosph\'erique lisse. Une conjecture plus faible affirme que le fibr\'e tangent d'une vari\'et\'e de Fano de nombre de Picard un est simple. 

Nous prouvons que cette conjecture plus faible est valable pour les variétés horosph\'eriques lisses de nombre de Picard   un. Notre preuve d\'ecoule de l'existence d'une famille irr\'eductible de courbes rationnelles inflexibles dont les vecteurs tangents engendrent les espaces tangents de la vari\'et\'e horosph\'erique en des points g\'en\'eraux.}

 \section{Introduction}

Tangent bundles of Fano manifolds have been studied from various   points of view,  in relation with the existence of a K\"ahler-Einstein metric.  For example,  if  $X$ admits  a K\"ahler-Einstein metric, then the Lie algebra of holomorphic vector fields on $X$ is reductive.
 A weaker condition, the existence of a Hermitian-Einstein metric on the tangent bundle $TX$ is equivalent to the polystability of $TX$, and is equivalent to the stability of $TX$ when $X$ has Picard number one.

 In this regard, when $X$ has Picard number one, the tangent bundle $TX$ of a Fano manifold $X$  has been expected to be semistable for a long time.

 \begin{conjecture} [Conjecture 0.1 of \cite{Ka20}]
 The tangent bundle of a Fano manifold of Picard number one is semistable.
 \end{conjecture}

 This conjecture has been verified in many cases (see e.g., \cite{PW95}, \cite{Hw98}).  Recently, however, a counterexample
has been discovered by Kanemitsu (\cite{Ka20}):
 The tangent bundle of a smooth horospherical variety $X$ of Picard number one   is not semistable if $X$ is of type $(B_n, \varpi_{n-1}, \varpi_n)$, where $n \geq 4$ or of type $(F_4, \varpi_3, \varpi_2)$, and is stable, otherwise. For the types of horospherical varieties of Picard number one, see Proposition \ref{prop:horospherical varieties pasquier}.

Only scalar multiplications are endomorphisms of a stable vector bundle, so that  stable vector bundles are simple. Instead of stability,
we consider a weaker conjecture.

\begin{conjecture} \label{conj:simplicity}  The tangent bundle of a Fano manifold of Picard number one is simple.
\end{conjecture}

 In this paper, we prove that Conjecture \ref{conj:simplicity} is valid for any smooth horospherical variety of Picard number one.

\begin{theorem} \label{thm:simplicity of horospherical}
The tangent bundle of any smooth horospherical varieties of Picard number one  is simple.
\end{theorem}

To prove Theorem \ref{thm:simplicity of horospherical} we use  unbendable rational curves. A rational curve $f:\mathbb P^1 \rightarrow X$ in a uniruled projective manifold $X$ is said to be {\it unbendable} if the pull-back $f^*TX$ of the tangent bundle $TX$ of $X$   is decomposed as $\mathcal O(2) \oplus \mathcal O(1)^p \oplus \mathcal O^q$ for some nonnegative integers $p,q$.
Minimal rational curves are unbendable but the converse is not true (Remark \ref{rem:unbendable but not minimal}). Here, by a minimal rational curve we mean a rational curve $f:\mathbb P^1 \rightarrow X$ with $f^*TX$ being nonnegative, and the degree of $f$ with respect to a fixed ample line bundle is minimal among all such rational curves.

We expect  unbendable rational components   will play  an important role in understanding the geometry of a uniruled projective manifold $X$ as minimal rational components did (see e.g., \cite{Hw12} and \cite{Mo08} and references therein).
When $X$ has Picard number one, we get the simplicity of the tangent bundle $TX$  if there is an irreducible family of unbendable rational curves whose tangent vectors span   $\mathbb PT_x X$ at a general point $x $  in $X$ (Proposition \ref{prop:nondegenerate unbendable curve implies simplicity}).

We prove that for a rational homogeneous variety $G/P$ with $G$ simple or a smooth horospherical variety of Picard number one,
 such an unbendable rational component exists
(Proposition \ref{pro:unbendable curves in rational homogeneous variety}
and Proposition \ref{pro:unbendable curves in smooth horospherical variety}).
Then Theorem \ref{thm:simplicity of horospherical} follows from Proposition \ref{prop:nondegenerate unbendable curve implies simplicity} and the condition that $X$ has Picard number one.

In relation with Conjecture \ref{conj:simplicity} we ask whether there is such an unbendable rational component  for any Fano manifold  of Picard number one.

\begin{conjecture} \label{conj:nondegenerate unbendable curves} A Fano manifold of Picard number one admits an irreducible family of unbendable rational curves
whose tangent vectors span the tangent spaces of the Fano manifold at general points.
\end{conjecture}

If Conjecture \ref{conj:nondegenerate unbendable curves} holds, then so does Conjecture \ref{conj:simplicity} by Proposition \ref{prop:nondegenerate  unbendable curve implies simplicity}. Minimal rational curves are unbendable curves, but   their  tangent vectors do not always generate  the tangent space at a general point.
Since a Fano manifold of Picard number one is rationally connected,  there  exists an irreducible family of rational curves whose tangent vectors span the tangent space  at a general point, but the restriction of the tangent bundle to rational curves in this family may have more than one $\mathcal O(2)$-factors, or more generally, may have $\mathcal O(a)$-factors with $a \geq 3$.
The question is whether we have an irreducible family of rational curves having both properties, rational curves in the family are unbendable and their tangent vectors  span the tangent space at a general point.


 \section{Unbendable rational curves} \label{sect:unbendable rational curve}

Let $X$ be a uniruled projective manifold.
We say that a rational curve $f:\mathbb P^1 \rightarrow X$ in  $X$ is
\begin{enumerate}
\item[$\bullet$] {\it unbendable} if the pull-back $f^*TX$ of the tangent bundle $TX$ of $X$   is decomposed as $\mathcal O(2) \oplus \mathcal O(1)^p \oplus \mathcal O^q$ for some nonnegative integers $p,q$;
\item[$\bullet$] {\it minimal} if the pull-back $f^*TX$ of the tangent bundle $TX$ of $X$  is  nonnegative, and the degree of $f$ with respect to a fixed ample line bundle is minimal among all such rational curves.
\end{enumerate}
 An irreducible component $\mathcal H$ of the Hilbert scheme of rational  curves in $X$ containing an unbendable (minimal, respectively) rational curve is called an {\it unbendable (minimal, respectively) rational component} of $X$.





Fix    an unbendable rational component $\mathcal H$ of $X$.
 For a   point $x \in X$ denote by $\mathcal H_x$ the subscheme  of $\mathcal H$ consisting of members passing through $x$, and  define a rational map $$\tau_x:\mathcal H_x \dasharrow \mathbb PT_x X$$ by sending a rational curve $[C] \in \mathcal H_x$ smooth at $x$ to its tangent direction $[T_xC] \in \mathbb PT_xX$. The closure $\mathcal C_x$ of the image of  $\tau_x$ is called the {\it variety of tangents} at $x$ of the family $\mathcal H$.

 \begin{proposition} [Theorem 2 of \cite{Hw02}] \label{prop:nondegenerate unbendable curve implies simplicity}
If $X$ has    an unbendable rational component $\mathcal H$
such that
\begin{enumerate}
\item   the variety of tangents at $x$ of $\mathcal H$ is nondegenerate  in the projective tangent space $\mathbb P T_xX$ for general $x \in X$;
\item   a general point of $X$ is joined to a point in $X$ by a connected chain of curves  in $\mathcal H$.
\end{enumerate}
then $TX$ is    simple, that is, any endomorphism of $TX$ is a scalar multiplication.

\end{proposition}

When $X$ has Picard number one, the condition (2) of Proposition \ref{prop:nondegenerate unbendable curve implies simplicity} is satisfied automatically:  There is a sequence of locally closed submanifolds $\mathcal X^0=\{x\} \subsetneq \mathcal X^1 \subsetneq \dots \subsetneq \mathcal X^m$, where $\dim \mathcal X^m = \dim X$, such that any point in $\mathcal X^k$ can be connected to a point in $\mathcal X^{k-1}$ by a rational curve in $\mathcal H$ for any $1 \leq k \leq m$ (see Section 4.3 of \cite{HM98} or Section 3 of \cite{Mo08}). 

For example, the moduli of semistable vector bundles of rank $r$ and with a fixed determinant  (\cite{Hw02}) or a wonderful group compactification (\cite{BF15}) have an unbendable rational component whose variety of tangents is nondegenerate in $\mathbb P(T_xX)$ at a general point $x \in X$. In the first case, in fact,  a  minimal rational component satisfies the desired property.  However, 
this is not the case  in general
(Remark \ref{rem:unbendable but not minimal}), and we need to consider an unbendable rational component  which is not minimal to get the desired property.

 \section{Existence of unbendable rational curves}
 \subsection{Rational homogeneous varieties}

 Let $\frak g$ be a complex simple Lie algebra. Let $\frak  h$ be a Cartan subalgebra and  $\Delta=\{\alpha_1, \dots, \alpha_{\ell}\}$ be a system of simple roots. For each root $\alpha$, there is an element $H_{\alpha}$ in $\frak h$ satisfying that $\alpha(H)= (H_{\alpha}, H)$ for any $H \in \frak h$, where $ (\,,\,)$ is the  Killing form of $\frak g$. Then the Killing form induces a symmetric bilinear form $(\,,\,)$ on the set $\Phi$ of roots defined  by $(\alpha, \beta):=(H_{\alpha}, H_{\beta})$ for $\alpha, \beta \in \Phi$.
 Put $h_{\alpha}:=2H_{\alpha}/ (\alpha, \alpha)$ and $\alpha^{\vee}:=2\alpha/(\alpha, \alpha)$ for $\alpha \in \Phi$. Then   $\beta(h_{\alpha})=(\beta, \alpha^{\vee})$.

For a subset $\Delta_1\subset \Delta$, define a function $n_{\Delta_1}: \Phi \rightarrow \mathbb Z$ by $$n_{\Delta_1}(\alpha) = \sum_{\alpha_i \in \Delta_1} n_i$$
for $\alpha = \sum_{i=1}^{\ell} n_i \alpha_i \in \Phi$. Define $\Phi_{\Delta_1}^{\pm} $ and $\Phi_{\Delta_1}^0$ by
$$\Phi_{\Delta_1}^{\pm}:= \{\alpha \in \Phi: \Phi: n_{\Delta_1}(\alpha) \in \mathbb Z_{\pm} \} \text{ and } \Phi_{\Delta_1}^0 :=\{ \alpha \in \Phi: n_{\Delta_1}(\alpha) =0\}  .$$
and put  $\frak p_{\Delta_1} =\frak h \oplus (\oplus_{\alpha \in \Phi_{\Delta_1}^0 \cup \Phi_{\Delta_1}^-} \frak g_{\alpha}   )$  and $\frak m_{\Delta_1} = \oplus_{\alpha \in \Phi_{\Delta_1}^+} \frak g_{\alpha}$.
Then $\Phi = \Phi_{\Delta_1}^- \cup \Phi_{\Delta_1}^0 \cup \Phi_{\Delta_1}^+$ and
$\frak g = \frak p_{\Delta_1} \oplus \frak m_{\Delta_1} $.
Fix $\mathfrak p=\mathfrak p_{\Delta_1}$ 
from now on. Note that $\mathfrak p$ contains the negative Borel subalgebra $\frak b := \frak p_{\Delta}$.

Let $G$ be a  simply connected algebraic group with Lie algebra $\frak g$ and $P$ be the subgroup of $G$ with Lie algebra $\frak p$.
 For each root $\alpha$,
 take $E_{\alpha} \in \frak g_{\alpha}, E_{-\alpha} \in \frak g_{-\alpha}$ such that $(E_{\alpha}, E_{-\alpha})=1$. Then  $ [E_{\alpha}, E_{-\alpha}]=H_{\alpha}  \in \frak h$ and $E_{\alpha}, E_{-\alpha}, H_{\alpha}$ generate a subalgebra $\frak{sl}_{\alpha}$ of $\frak g$ which is isomorphic to  $\frak{sl}_2$.
 Let
 $S_{\alpha}$ be the subgroup of $G$ with Lie algebra $\frak {sl}_{\alpha}$.
    The $S_{\alpha}$-orbit of $o \in G/P$  is a rational curve in $G/P$,  denoted by  $C_{\alpha}$.

    The tangent bundle $T(G/P)$ is the homogeneous vector bundle on $G/P$ associated to the representation $P \rightarrow GL(\frak g/\frak p)$.
      The tangent bundle  $T(G/P)$  restricted to $C_{\alpha}$ is decomposed as $\oplus_{\beta \in \Phi_{\Delta_1}^+} \mathcal O(\beta(h_{\alpha}))=\oplus_{\beta \in \Phi_{\Delta_1}^+} \mathcal O((\beta,\alpha ^{\vee}))$ (See the proof of Proposition 2 of \cite{HM02}).

 \begin{proposition} \label{pro:unbendable curves in rational homogeneous variety}  Let $ G/P$ be a  rational homogeneous variety. Assume that $G$ is simple.
 Then there is an unbendable rational component $\mathcal H$ whose variety $\mathcal C_y$ of    tangents at any $y \in   G/P$ is nondegenerate in $\mathbb PT_y (G/P)$.
 \end{proposition}

   \begin{proof}
   Let $\theta$ be the maximal positive root of $\frak g$. Since $\theta$ has nonzero coefficient in $\alpha_i$ for any $i=1, \dots,\ell$, $S_{\theta}$ is not contained in $  P_{\Delta_1}$ for any subset $\Delta_1 $ of $\Delta$, and thus $C_{\theta}$ is a (nonconstant) rational curve in $G/P$ for any $P=P_{\Delta_1}$. We will show that $C_{\theta}$ is an unbendable rational curve in $G/P$. It suffices to show that $(\beta, \theta^{\vee}) \leq 1$ for all $\beta \in \Phi^+_{\Delta_1} \backslash \{\theta\}$. To do this we will use the following two facts:
   \begin{enumerate}
   \item[(i)] If $G$ is not of type $A$, then there is $i_0 \in \{1, \dots, \ell\}$ such that  $(\alpha_i, \theta)  =0$ for all $i \not =i_0 $ in $\{1, \dots, \ell\}$   and $(\alpha_{i_0}, \theta^{\vee}) = 1$.
   \item[(ii)] If $G$ is of type $A$, then $(\alpha_j, \theta)=0$ for all $j   \in \{2, \dots, \ell-1\}$ and $(\alpha_1, \theta^{\vee}) = (\alpha_{\ell}, \theta^{\vee}) =1$.
   \end{enumerate}
Following the conventions in \cite{OV90} for the indices of simple roots,    we have $i_0=2$ if $G$ is of type $B_{\ell},  D_{\ell}$, $i_0=1$ if $G$ is of type $C_{\ell}$,  and $i_0=6$ ($ 6,1,4,2$, respectively) if $G$ is of type $E_6$ ($E_7, E_8, F_4, G_2$, respectively).   See the affine Dynkin diagram of $\Delta^{(1)}$ in Table 6 of \cite{OV90}, which  is the diagram of the union $\Delta \cup\{-\theta\}$ of  $\Delta$ in Table 1 of \cite{OV90} and  $\{-\theta\}$.  \\

    If $G$ is not of type $A$, then, by (i), the maximal root among roots in $\Phi \backslash \{\theta\}$ is $\theta-\alpha_i$.
  Since $( \theta-\alpha_i, \theta ^{\vee}) = 2-1=1$,  we have $(\beta, \theta^{\vee}) \leq 1$ for all $\beta \in \Phi^+_{\Delta_1} \backslash \{\theta\}$.
  If $G$ is of type $A$, then, by (ii), among roots in $\Phi \backslash \{\theta\}$, maximal ones are either $\theta-\alpha_1$ or $\theta-\alpha_{\ell}$. From $(\theta-\alpha_1, \theta^{\vee})=(\theta-\alpha_{\ell}, \theta^{\vee}) = 1$ it follows that $(\beta, \theta^{\vee}) \leq 1$ for all $\beta \in \Phi^+_{\Delta_1} \backslash \{\theta\}$.

The curve $C_{\theta}$ is tangent to $\mathbb C E_{\theta} \subset T_{o}G/P$. The variety $\mathcal C_{o}$ of tangents at the base point $o$ is the closure of the $P$-orbit  $P.[E_{\theta}]$ in $ \mathbb P(\frak m )$. We claim that $\mathcal C_{o}$ is nondegenerate in $\mathbb P(\frak g/\frak p )$.
Consider the projection $\frak g \rightarrow \frak g/\frak p   $. The closure of the $P$-orbit $P.[E_{\theta}]$ in $\mathbb P(\frak g)$ is the $G$-orbit $G.[E_{\theta}]$, which is nondegenerate in $\mathbb P(\frak g)$ because it is the highest weight orbit. Therefore, the closure of the $P$-orbit $P.[E_{\theta}]$ in $\mathbb P(\frak g/\frak p  )$ is nondegenerate in $\mathbb P(\frak g/\frak p  )$.
\end{proof}

\begin{remark}
 The inequality $(\beta, \theta^{\vee}) \leq 1$ for any $\beta \in \Phi^+\backslash \{ \theta\}$ also follows from Proposition 25 in  VI.1.8   of \cite{Bo07}.
\end{remark}

\begin{remark} \label{rem:unbendable but not minimal} Take a simple root $\alpha_i$ which is not short and consider the rational homogenous variety $G/P$ associated with $\Delta_1=\Delta -\{\alpha_i\}$.  Then the rational curve $C_{\alpha_i}$ is a minimal rational curve and the variety of tangents of the minimal rational component containing $[C_{\alpha_i}]$ at a point $x \in G/P$  spans a proper subspace of $\mathbb P T_x(G/P)$ if the coefficient of $\alpha_i$ in $\theta$ is $> 1$ (Proposition 1 of \cite{HM02}). Therefore, an unbendable rational curve is not necessarily a minimal rational curve.
\end{remark}

Together with Proposition \ref{prop:nondegenerate unbendable curve implies simplicity},  Proposition \ref{pro:unbendable curves in rational homogeneous variety} reproves the following result.
\begin{proposition} [Theorem 2.1 of \cite{AzBi10}] Let $X=G/P$ be a rational homogeneous variety with $G$ being simple.
Then the tangent bundle of $X$ is simple.
\end{proposition}

\begin{proof}
By Proposition \ref{prop:nondegenerate unbendable curve implies simplicity} and Proposition \ref{pro:unbendable curves in rational homogeneous variety}, it suffices to show that a general point $x$ in $G/P$ is joined to the base point $o$ by a connected chain of curves in the unbendable rational component $\mathcal H$ constructed in Proposition \ref{pro:unbendable curves in rational homogeneous variety}.

Let $\Sigma$ denote the subset of $G/P$ consisting of $x \in G/P$ which are joined to $o$ by a connected chain of curves in $\mathcal H$. Then   the stabilizer $Q$ of $\Sigma$ in $G$ contains the subgroup of $G$ generated by $P$ and $S_{\theta}$ because the action of $S_{\theta}$ moves the point $o$ to a point in $C_{\theta}$.  Thus $Q$ contains $P$ properly.

Let $\{\varpi_1, \dots, \varpi_{\ell}\}$ denote the   system of fundamental weights corresponding the system $\Delta$ of simple roots. Then each $\varpi_j$ can be written as a linear combination of $\alpha_i$'s with positive coefficients (See Table 2 of \cite{OV90}). By (i) and (ii) in the proof of Proposition \ref{pro:unbendable curves in rational homogeneous variety}, $(\varpi_j, \theta^{\vee})$ is positive for any $j=1, \dots, \ell$. Therefore, $C_{\theta}$ cannot be contained in a fiber of a nontrivial projection $G/P \rightarrow G/Q$. Since $Q/P$ is not a point, $G/Q$ is a point, that is, $G$ is $Q$ and $\Sigma$ is $G/P$.
\end{proof}

\subsection{Smooth horospherical varieties of Picard number one}
    For $i=1, \dots, \ell$, let $V_{\varpi_i}  $ be the irreducible representation of $G$ of highest weight  $\varpi_i $ and let $v_{\varpi_i}$ be a highest weight vector  of  $V_{\varpi_i}$.
  Let $(G, \varpi_i, \varpi_j)$ denote the closure of the $G$-orbit of the sum $[v_{\varpi_i}+ v_{\varpi_j}]$ in $\mathbb P(V_{\varpi_i} \oplus V_{\varpi_j})$.

\begin{proposition}[\cite{Pa09}] \label{prop:horospherical varieties pasquier}
Let $X$ be a smooth horospherical variety of Picard number one. Then, $X$ is either a rational homogeneous variety    or one of the following:

\begin{enumerate}
\item $(B_n, \varpi_{n-1}, \varpi_n)$ ($n \geq 3$)
\item $(B_3, \varpi_1, \varpi_3)$
\item $(C_n, \varpi_{k }, \varpi_{k-1})$ ($n \geq 2$, $2 \leq k \leq n$)
\item $(F_4, \varpi_3, \varpi_2)$
\item $(G_2, \varpi_2, \varpi_1)$.
\end{enumerate}
In the latter case, the automorphism group $\Aut^0(X)$ is given by $\widetilde{G} \ltimes H^0(G/P, G \times_PV)$, where $\widetilde{G}$ is a reductive group with $G$ a maximal semisimple subgroup,  and  the open $\Aut^0(X)$-orbit $\mathcal O$  in $X$ is $G$-equivariantly isomorphic to      the total space of a homogeneous vector bundle $G\times_PV$ on $G/P$, where $P$ is the maximal parabolic subgroup associated to
$$\varpi_{n-1}, \,\, \varpi_1, \,\,\varpi_{k-1},\,\,\varpi_3,\,\, \varpi_2, \text{ respectively }$$
and $V$ is the simple $P$-module of highest weight $\lambda_V$ given by
$$\varpi_{n-1} - \varpi_{n}, \,\,\varpi_1- \varpi_3, \,\,\varpi_{k }-\varpi_{k-1},\,\, \varpi_3-\varpi_2, \,\,\varpi_2-\varpi_1,   \text{ respectively }.$$
\end{proposition}

\begin{remark}
We remark that  our convention is different from that of \cite{Pa09}.
The way of indexing simple roots and fundamental weights  is the same as in \cite{OV90} in this paper while \cite{Pa09} follows the convention in \cite{Bo07}.
The isotropy subgroup $P$ contains the negative Borel subgroup in this paper while it contains the positive Borel subgroup in \cite{Pa09}.
\end{remark}

 \begin{proposition} \label{pro:unbendable curves in smooth horospherical variety}  Let $ X$ be a  smooth horospherical variety of Picard number one.
 Then there is an unbendable rational component $\mathcal H$ whose variety $\mathcal C_x$ of   tangents at any $x $ in the open $\Aut^0(X)$-orbit $\mathcal O$ is nondegenerate in $\mathbb PT_x X$.
 \end{proposition}

\begin{proof}
We will show that there is an unbendable rational component $\mathcal H$ such that  for any point $x $ in the open $\Aut^0(X)$-orbit $\mathcal O$, the variety of tangents of $\mathcal H$ at $x$ is nondegenerate in $\mathbb P T_xX$.  Since $G \times H^0(G/P, G \times_PV)$ acts on $\mathcal O = G\times_PV$ transitively, we may assume that $x=[o,0]$, where $o$ is the base point in $G/P$ and $0$ is the zero element  in $V$. Identify $G/P$ with the zero section $Y$ of $G \times_PV$ and let $C_{\theta}$ be the rational curve in $G/P \simeq Y$ constructed in Proposition \ref{pro:unbendable curves in rational homogeneous variety}.  We claim that  $C_{\theta}$ is   an unbendable curve in $\mathcal O$.

Let $\lambda_V$ be the highest weight  of $V$ listed in Proposition \ref{prop:horospherical varieties pasquier}.
By the facts (i) and (ii) in the proof of Proposition \ref{pro:unbendable curves in rational homogeneous variety}, $( \lambda_V, \theta^{\vee}) $ is the coefficient of $\alpha_{i_0}$ in the expression of  $\lambda_V$ via simple roots, $\alpha_1, \dots, \alpha_{\ell}$.
For the description of $\varpi_i$ as a linear combination of simple roots, see   Table 2 in \cite{OV90}: The $i$-th column of the matrix $(A^t)^{-1}$ inverse to the transposed Cartan matrix $A$ is the list of coefficients of simple roots in $\varpi_i$. Let $ b_{i,j} $ denote the $(i,j)$-th element  of the matrix $(A^t)^{-1}$.
  Then   $( \lambda_V, \theta^{\vee})$ is given by
  $$ b_{2, n-1} - b_{2,n}, \quad   b_{2,1} -b_{2,3}, \quad b_{1,k} - b_{1,k-1}, \quad b_{4,3} -b_{4,2}, \quad b_{2,2}- b_{2,1}, \quad \text{respectively}. $$
  Using the description of the matrix $(A^t)^{-1}=(b_{i,j})$  in Table 2 of \cite{OV90} we get that  $( \lambda_V, \theta^{\vee})$ is given by
  $$\frac{1}{2}(4-2), \quad \frac{1}{2}(2-2) ,\quad \frac{1}{2}(4-4),\quad 3-2, \quad2-1,  \quad \text{respectively}.   $$
 %
%
%
 %
Therefore, $( \lambda_V, \theta^{\vee}) $ is $1$ or $0$, and thus
the vector bundle $G \times_P V$ restricted to $C_{\theta} \subset G/P$ splits as $\oplus_{i=1}^m\mathcal O(a_i)$, where $$1 \geq a_1 \geq \cdots \geq a_m \geq 0. $$
  From this property and  the nonnegativity of $TY|_{C_{\theta}}$, it follows that the short exact sequence $0 \rightarrow TY \rightarrow TX|_Y \rightarrow G \times_P V \rightarrow 0$  restricted to $C_{\theta}$ splits. Since $TY|_{C_{\theta}}$   has only one $\mathcal O(2)$-factor by Proposition \ref{pro:unbendable curves in rational homogeneous  variety},  we get that   $TX|_{C_{\theta}}$ is also nonnegative and has only one $\mathcal O(2)$-factor.

Now
let $\mathcal V$ be the vector space  spanned by $$\{ T_{[o,0]}(a C_{\theta}): [o,0] \in a C_{\theta}, a \in G \ltimes H^0(G/P, G \times_PV) \}.$$
Then the projection of $\mathcal V$ to $V$ is a   nonzero $P$-stable subspace of $V$,   and thus is the whole $V$ because $V$ is an irreducible $P$-module. Hence  the tangent directions to translates of $C_{\theta}$ passing through $[o,0]$ span  $T_{[o,0]}\mathcal O$.
\end{proof}

\begin{proof} [Proof of Theorem \ref{thm:simplicity of horospherical}]
 By Proposition \ref{pro:unbendable curves in smooth horospherical variety}, there is an unbendable rational component $\mathcal H$ whose variety of tangents is nondegenerate in $\mathbb P(T_xX)$ for any $x \in \mathcal O$.
  Since $X$ has Picard number one, 
  by Proposition \ref{prop:nondegenerate unbendable curve implies simplicity},
  any endomorphism of the tangent bundle of $X$ is a scalar multiple.
\end{proof}

\end{document}